\documentclass[12pt, reqno]{amsart}
\usepackage{amsmath, amsthm, amscd, amsfonts, amssymb, graphicx, color}
\usepackage[bookmarksnumbered, colorlinks, plainpages]{hyperref}

\usepackage{enumerate}
\usepackage{mathrsfs}

\newtheorem{theorem}{Theorem}[section]
\newtheorem{lemma}[theorem]{Lemma}
\newtheorem{proposition}[theorem]{Proposition}
\newtheorem{corollary}[theorem]{Corollary}
\theoremstyle{definition}
\newtheorem{definition}[theorem]{Definition}
\newtheorem{example}[theorem]{Example}

\theoremstyle{remark}
\newtheorem{remark}[theorem]{Remark}
\numberwithin{equation}{section}

\newcommand{\e}{\varepsilon}
\newcommand{\vp}{\varphi}
\DeclareMathOperator{\mis}{\mathfrak{M}}
\DeclareMathOperator{\Lip}{Lip}
\DeclareMathOperator{\Inv}{Inv}

\newcommand{\C}{\mathbb{C}}
\newcommand{\N}{\mathbb{N}}
\newcommand{\U}{\mathbf{1}}
\newcommand{\Sp}{\text{\textsc{sp}}}
\newcommand{\vsp}{\vec{\text{\textsc{sp}}}}

\newcommand{\tp}{\otimes}
\newcommand{\ptp}{\mathbin{\wh{\otimes}}}
\newcommand{\wh}{\widehat}

\newcommand{\A}{\mathscr{A}}
\newcommand{\fA}{\mathfrak{A}}
\newcommand{\cB}{\mathcal{B}}
\newcommand{\cF}{\mathcal{F}}
\newcommand{\cE}{\mathcal{E}}
\newcommand{\cG}{\mathcal{G}}
\newcommand{\bfa}{\mathbf{a}}

\newcommand{\set}[1]{\{#1\}}
\newcommand{\bigset}[1]{\bigl\{ #1 \bigr\}}
\newcommand{\Bigset}[1]{\Bigl\{ #1 \Bigr\}}
\newcommand{\biggset}[1]{\biggl\{ #1 \biggr\}}

\newcommand{\abs}[1]{\left\lvert #1 \right\rvert}
\newcommand{\bigabs}[1]{\bigl\lvert #1 \bigr\rvert}
\newcommand{\Bigabs}[1]{\Bigl\lvert #1 \Bigr\rvert}

\newcommand{\prn}[1]{(#1)}
\newcommand{\bigprn}[1]{\bigl( #1 \bigr)}
\newcommand{\Bigprn}[1]{\Bigl( #1 \Bigr)}
\newcommand{\biggprn}[1]{\biggl( #1 \biggr)}

\newcommand{\enorm}{\lVert\,\cdot\,\rVert}
\newcommand{\norm}[1]{\lVert #1 \rVert}
\newcommand{\lVertt}{\lvert\mspace{-2mu}\lvert\mspace{-2mu}\lvert}
\newcommand{\rVertt}{\rvert\mspace{-2mu}\rvert\mspace{-2mu}\rvert}
\newcommand{\enormm}{\lVertt\, \cdot \,\rVertt}
\newcommand{\normm}[1]{\lVertt #1 \rVertt}

\begin{document}
\setcounter{page}{1}

\title[Vector-valued spectra]%
{Vector-valued spectra of Banach algebra valued continuous functions}

\author[M. Abtahi, S. Farhangi]{Mortaza Abtahi$^1$ and Sara Farhangi$^2$}

\address{$^{1}$ School of Mathematics and Computer Sciences,
Damghan University, Damghan, P.O.BOX 36715-364, Iran.}
\email{\textcolor[rgb]{0.00,0.00,0.84}{abtahi@du.ac.ir; mortaza.abtahi@gmail.com}}

\address{$^{2}$ School of Mathematics and Computer Sciences,
Damghan University, Damghan, P.O.BOX 36715-364, Iran.}
\email{\textcolor[rgb]{0.00,0.00,0.84}{mehmandost@std.du.ac.ir; sfarmdst@gmail.com}}


\subjclass[2010]{Primary 46J10; Secondary 46J20.}

\keywords{Commutative Banach algebras, Banach function algebras,
Vector-valued function algebras, Vector-valued characters, Vector-valued spectra.}


\begin{abstract}
  Given a compact space $X$, a commutative Banach algebra $A$,
  and an $A$-valued function algebra $\A$ on $X$, the notions of
  vector-valued spectrum of functions $f\in\A$ are discussed.
  The $A$-valued spectrum $\vsp_A(f)$ of every $f\in\A$ is defined
  in such a way that $f(X) \subset \vsp_A(f)$. Utilizing the
  $A$-characters introduced in (M. Abtahi, \textit{Vector-valued characters
  on vector-valued function algebras}, Banach J. Math. Anal.
  (\emph{to appear}) \texttt{arXiv:1509.09215}), it is proved that
  $\vsp_A(f) = \set{\Psi(f):\text{$\Psi$ is an $A$-character of $\A$}}$.
  For the so-called natural $A$-valued function algebras, such as $C(X,A)$ and
  $\Lip(X,A)$, we see that $\vsp_A(f)=f(X)$. When $A = \C$, Banach $A$-valued
  function algebras reduce to Banach function algebras, $A$-characters reduce
  to characters, and $A$-valued spectra reduce to usual spectra.
\end{abstract}

\maketitle

\section{Introduction}

Let $A$ be a commutative complex Banach algebra with a unit element $\U$, and let
$\mis(A)$ denote the character space (maximal ideal space) of $A$.
Given $a\in A$, the spectrum $\Sp(a)$ of $a$ consists of those complex numbers $\lambda$ for which
$\lambda\U-a$ is not invertible in $A$. It is known that $\Sp(a)$ is compact and
\begin{equation}\label{eqn:intro:sp(a)=phi(a)}
  \Sp(a)=\set{\phi(a):\phi\in\mis(A)}.
\end{equation}

\subsection{Vector-valued Spectra}
Let $n$ be a positive integer and $\bfa=(a_1,\dotsc,a_n)$ be an $n$-tuples of elements of $A$.
The \emph{joint spectrum} of $\bfa$, again denoted by $\Sp(\bfa)$, is defined to be
the set of all $n$-tuples $(\lambda_1,\dotsc,\lambda_n)$ of scalars in $\C$
such that the unit element $\U$ does not belong to the ideal generated of $A$
by $\set{\lambda_i\U-a_i:1\leq i \leq n}$. That is,
\begin{equation}\label{eqn:intro:Sp(bfa)-dfn}
  \Sp(\bfa)=\Bigset{(\lambda_1,\dotsc,\lambda_n)\in\C^n:
    \U \notin \sum_{i=1}^n A(\lambda_i\U-a_i)}.
\end{equation}

\noindent
It is proved (e.g. \cite{Gamelin-UA}) that $\Sp(\bfa)$
is a compact set in $\C^n$ and
\begin{equation}\label{eqn:intro:Sp(bfa)-thm}
  \Sp(\bfa)=\bigset{(\phi(a_1),\dotsc,\phi(a_n)):
    \phi\in \mis(A)}.
\end{equation}

If one sets $X=\set{1,\dotsc,n}$, every $n$-tuples $\bfa=(a_1,\dotsc,a_n)$ can be seen as
an $A$-valued function $\bfa:X\to A$, $i\mapsto a_i$.
For every $\phi\in\mis(A)$, the composition of $\bfa:X\to A$ and $\phi:A\to\C$
gives $\phi\circ \bfa:X\to\C$, $i\mapsto \phi(a_i)$. With this convention,
equality \eqref{eqn:intro:Sp(bfa)-thm} is rephrased as follows;
\begin{equation}\label{eqn:intro:Sp(bfa)=phi o bfa}
  \Sp(\bfa)=\bigset{\phi\circ \bfa: \phi\in \mis(A)}.
\end{equation}

In general, let $X$ be a nonempty set, and let $f:X\to A$ be a function.
The \emph{vector-valued spectrum} of $f$ is defined to be
\begin{equation}\label{eqn:intro:Sp(f)-dfn}
  \vsp(f) = \Bigset{\lambda:X\to \C: \U\notin \sum_{x\in F} A\bigprn{\lambda(x)\U-f(x)}},
\end{equation}
where $F$ runs over finite subsets of $X$. To prevent any confusion and
to distinguish between different types of spectrum, vector-valued spectrum
is denoted $\vsp(f)$.

Form \eqref{eqn:intro:Sp(f)-dfn} a function $\lambda:X\to \C$ does not belong to $\vsp(f)$
if, and only if, there exist $x_1,\dotsc,x_n$ in $X$ and
$a_1,\dotsc,a_n$ in $A$ such that
\begin{equation}\label{eqn:intro:not-in-Sp(f)}
  \U=\sum_{i=1}^n a_i\bigprn{\lambda(x_i)\U-f(x_i)}.
\end{equation}

\noindent
Extending \eqref{eqn:intro:Sp(bfa)=phi o bfa},
later in Theorem \ref{thm:SP(f)=phi o f}, we establish the following equality;
\begin{equation}\label{eqn:intro:SP(f)=phi o f}
   \vsp(f)=\set{\phi\circ f : \phi \in \mis(A)}.
\end{equation}

It is then clear that if $f:X\to A$ is continuous, then $\vsp(f)\subset C(X)$.
In this case, we will see (Theorem \ref{thm:Sp(f) is compact in C(X)})
that $\vsp(f)$ is a compact subset of $C(X)$. In general, if $X$ is enriched
with some structure (topological, algebraical, etc.),
and $f:X\to A$ is an appropriate morphism, then
many structural properties of $f$ are inherited by every $\lambda\in \vsp(f)$.

\begin{proposition}[\cite{Dineen-Harte-Taylor-I}, \cite{Harte-gmj}]
  Let $f:X\to A$ be a function and $\lambda:X\to \C$ be in the vector-valued spectrum
  $\vsp(f)$.
  \begin{enumerate}[\quad$(1)$]
    \item If $f$ is bounded, then so is $\lambda$.

    \item If $X$ is a topological space and $f\in C(X,A)$, then $\lambda\in C(X)$.

    \item If $X\subset \C^n$ and $f\in H_0(X,A)$, i.e.\ $f$ is holomorphic on
    a neighbourhood of $X$, then $\lambda\in H_0(X)$.

    \item If $(X,d\,)$ is a metric space and $f\in\Lip(X,A)$, then $\lambda\in\Lip(X)$.

    \item If $X=\N$ and $f\in\ell^1(\N,A)$, then $\lambda\in\ell^1(\N)$.

    \item If $X$ is a linear space and $f$ is linear, then so is $\lambda$.

    \item If $X$ is a Banach space and $f\in\cB(X,A)$, then $\lambda \in X^*$ and $\|\lambda\| \leq \|f\|$.

    \item If $X=\fA$ is a Banach algebra and $f$ is an algebra homomorphism,
    then $\lambda\in\mis(\fA)$.

    \item If $I:A\to A$ is the identity operator, then $\vsp(I)=\mis(A)$.
  \end{enumerate}
\end{proposition}

\subsection{The $A$-valued Spectrum}
We take a different approach to studying vector-valued spectrum.
To provide a motivation, assume $\fA$ is a complex function algebra on $X$.
For every $x\in X$, the evaluation homomorphism $\e_x:f\mapsto f(x)$
is a character of $\fA$ whence the spectrum $\Sp(f)$ contains the range $f(X)$.
In case $\fA$ is natural, i.e.\ $\e_x\ (x\in X)$ are the only
characters of $\fA$, we get $\Sp(f)=f(X)$.

Let $\A$ be an $A$-valued function algebra on $X$. For every $f\in\A$,
an $A$-valued spectrum $\vsp_A(f)$ of $f$ will be defined in such a way that
$f(X)\subset \vsp_A(f)$. Utilizing the $A$-valued characters \cite{Abtahi-vector-valued},
the following analogy of \eqref{eqn:intro:sp(a)=phi(a)}
will be established:
\begin{equation}\label{eqn:intro:SPA(f)=Psi(f)}
  \vsp_A(f) = \set{\Psi(f):\text{$\Psi$ is an $A$-character of $\A$}}.
\end{equation}

\subsection{Historical Background}

Many different spectra have been defined over the last seventy years.
The classical definition was given and developed in the 1950's by
Arens and Calder\'on \cite{Arens-Calderon}, Silov \cite{Shilov-1953}
and Waelbroeck \cite{Waelbroeck-1954} for $n$-tuples in a commutative
unital Banach algebra. Waelbroeck was
the first to consider the joint spectrum
of an infinite set of elements in a way that led to a functional calculus
for norm continuous holomorphic germs, \cite{Dineen-2004}.
Subsequently, his results were extended in various ways
by Matos \cite{Matos-1984, Matos-1988} and others.

When $A$ is a commutative Banach algebra and $\fA$ is a Banach space, Waelbroeck
defined an $\fA$-valued spectrum of an element $f$ of the projective
tensor product $\fA\ptp A$. In \cite{Matos-1984}, Matos considered the spectrum
as lying in $C(\mis(A),\fA)$, where $A$ is a uniform algebra with maximal ideal space
$\mis(A)$ and $\fA$ is a locally convex space with the approximation property.
In \cite{Matos-1988} he considered the spectrum when $\fA$ has an unconditional basis
and $A$ is an arbitrary commutative unital Banach algebra. In all these articles
an infinite-dimensional holomorphic functional calculus is developed.

In \cite{Harte-1973}, Harte defined, using left and right invertibility,
the joint spectrum of a system of elements in an arbitrary Banach algebra and,
subject to certain commutativity hypotheses, obtained spectral mapping theorems.
In \cite{Dineen-Harte-Taylor-I},
Dineen, Harte and Taylor defined a non-commutative
version of the Waelbroeck spectrum for tensor product elements in $\fA\ptp_\gamma A$,
where $\gamma$ is a uniform cross-norm, and by specialising to the case where $\fA$ was
itself a unital Banach algebra obtained a number of applications.
In \cite{Dineen-Harte-Taylor-II,Dineen-Harte-Taylor-III},
they continue this investigation and examine the behaviour
of the spectrum under polynomials and holomorphic mappings between Banach spaces.

\section{Preliminaries}

\subsection{Notations and Conventions}
  Throughout, $X$ is a compact Hausdorff space and $A$ is a commutative
  \emph{semisimple} Banach algebra with a unit element $\U$. The set of
  invertible elements of $A$ is denoted by $\Inv(A)$. The algebra of all
  continuous $A$-valued functions is denoted by $C(X,A)$. The uniform norm
  $\|f\|_X$ of a function $f\in C(X,A)$ is defined in the obvious way.

  If $f:X\to\C$ is a function and $a\in A$, we write $fa$ to denote the $A$-valued function
  $X\to A$, $x\mapsto f(x)a$. If $\fA$ is an algebra of complex-valued functions on $X$, we let $\fA A$ be
    the linear span of $\set{fa:f\in\fA,\,a\in A}$. Hence, an element $f\in \fA A$
    is of the form $f=f_1a_1+\dotsb+f_na_n$ with $f_j\in \fA$ and $a_j\in A$.

  Given an element $a\in A$, we use the same notation $a$ for the constant
  function $X\to A$ given by $a(x)=a$, for all $x\in X$, and consider $A$ as a closed
  subalgebra of $C(X,A)$. We identify $\C$ with the closed subalgebra $\C\U$ of $A$.
  Hence, every $\C$-valued function can be seen as an $A$-valued function.
  Given $f:X\to\C$, we use the same notation $f$ for the function
  $X\to A$, $x\mapsto f(x)\U$. We regard $C(X)$ as a closed subalgebra of $C(X,A)$.

  To every continuous function $f:X\to A$, we correspond the function
  \[
     \tilde f:\mis(A)\to C(X),\quad \tilde f(\phi)=\phi\circ f.
  \]

\subsection{Admissible Vector-valued Function Algebras}
An \emph{$A$-valued function algebra} on $X$ is a subalgebra $\A$ of $C(X,A)$ that
contains the constant functions $X\to A$, $x\mapsto a$,
with $a\in A$, and separates the points of $X$.
If $\A$ is endowed with some complete algebra norm $\enormm$ such that
the restriction of $\enormm$ to $A$ is equivalent to the original norm of $A$,
and $\|f\|_X \leq \normm{f}$, for every $f\in \A$, then $\A$ is called
a \emph{Banach $A$-valued function algebra} on $X$.
If the given norm is equivalent to the uniform norm $\enorm_X$,
$\A$ is called an \emph{$A$-valued uniform algebra}.
If no confusion can arise, instead of $\enormm$, we use the same notation $\enorm$ for the norm
of $\A$.

\begin{definition}[\cite{Abtahi-vector-valued}]
  An $A$-valued function algebra $\A$ is said to be \emph{admissible} if
  \begin{equation}\label{eqn:admissible-A-valued-FA}
     \set{(\phi\circ f)\U: \phi\in\mis(A), f\in \A} \subset \A.
  \end{equation}
\end{definition}

Admissible vector-valued function algebras are abundant. Some examples are
as follows. The algebra $C(X,A)$ of all continuous $A$-valued functions
is admissible. If $\fA$ is a complex function algebra on $X$
then $\fA A$ is an admissible $A$-valued function algebra on $X$, and thus
its uniform closure in $C(X,A)$ is an admissible $A$-valued uniform algebra.
More generally, the tensor product $\fA\tp A$ can be seen as an admissible
$A$-valued function algebra, and if $\gamma$ is a cross-norm on $\fA\tp A$,
its completion $\fA\ptp_\gamma A$ forms an admissible Banach $A$-valued function
algebra. 

Another example of an admissible Banach $A$-valued function algebra is
the $A$-valued Lipschitz function algebra $\Lip(X,A)$, where $X$
is a compact metric space (Example \ref{exa:LIP}).
An example in \cite{Abtahi-vector-valued}
shows that not all vector-valued function algebras are admissible.

During the paper, we let $\A$ be admissible and $\fA=\A \cap C(X)$, more precisely
$\fA=\A\cap C(X)\U$, be the subalgebra
of $\A$ consisting of all complex-valued functions in $\A$. Then $\fA$ forms
a complex-valued function algebra by itself and $\phi[\A]=\fA$, for all $\phi\in\mis(A)$.

\subsection{Vector-valued Characters}
\label{sec:vector-valued-characters}

Vector-valued characters are an obvious generalization of characters.
For every $x\in X$, define
$\cE_x:\A\to A$ by $\cE_x(f)=f(x)$. We call $\cE_x$ the \emph{evaluation homomorphism}
at the point $x$. Point evaluation homomorphisms are good examples of vector-valued characters.

\begin{definition}[\cite{Abtahi-vector-valued}]
\label{dfn:vector-valued-character}
  An \emph{$A$-character} of $\A$ is an algebra homomorphism
  $\Psi:\A\to A$ such that $\Psi(\U)=\U$ and $\phi(\Psi f)=\Psi(\phi\circ f)$,
  for all $f\in \A$ and $\phi\in\mis(A)$. The set of all $A$-characters of $\A$
  is denoted by $\mis_A(\A)$.
\end{definition}

If $\Psi:\A\to A$ is an $A$-character, then $\psi=\Psi|_\fA$ is a character of
$\fA$. If $\Psi_1$ and $\Psi_2$ are $A$-characters of $\A$ with $\psi_1=\Psi_1|_\fA$ and $\psi_2=\Psi_2|_\fA$,
then \cite{Abtahi-vector-valued}
\begin{equation}\label{eqn:Psi1=Psi2-iff}
   \Psi_1=\Psi_2 \Longleftrightarrow
   \ker\Psi_1=\ker\Psi_2 \Longleftrightarrow
   \ker\psi_1=\ker\psi_2 \Longleftrightarrow
   \psi_1=\psi_2.
\end{equation}

\begin{definition}[\cite{Abtahi-vector-valued}]
  Given a character $\psi\in\mis(\fA)$, if there exists an $A$-character $\Psi$ on $\A$
  such that $\Psi|_\fA=\psi$, then we say that $\psi$ \emph{lifts} to the $A$-character $\Psi$.
\end{definition}

By \eqref{eqn:Psi1=Psi2-iff}, if $\psi\in\mis(\fA)$ lifts to $\Psi_1$ and $\Psi_2$,
then $\Psi_1=\Psi_2$. For every $x\in X$, the unique $A$-character to which
the evaluation character $\e_x$ lifts is the evaluation homomorphism $\cE_x$.
Conditions under which every character $\psi$ lifts to
some $A$-character $\Psi$ are given in the following.

\begin{theorem}[\cite{Abtahi-vector-valued}]
\label{thm:every-psi-lifts-to-Psi-iff-every-f-extends-to-F}
  Let $E$ be the linear span of $\mis(A)$ in $A^*$.
  The following statements are equivalent.
  \begin{enumerate}[\upshape(i)]
    \item \label{item:g-is-continuous}
    For every $\psi\in\mis(\fA)$ and $f\in\A$, the mapping $g:E\to\C$, defined
    by $g(\phi)=\psi(\phi\circ f)$, is continuous with respect to the weak* topology of $E$.

    \item \label{item:every-psi-lifts-to-Psi}
    Every $\psi\in\mis(\fA)$ lifts to an $A$-character $\Psi\in \mis_A(\A)$.

    \item \label{item:every-f-extends-to-F}
    Every $f\in \A$ has a unique extension $F:\mis(\fA)\to A$ such that
    \[
      \phi(F(\psi)) = \psi(\phi\circ f) \quad (\psi\in \mis(\fA),\, \phi\in\mis(A)).
    \]
  \end{enumerate}
\end{theorem}

We recall that $\fA=\A\cap C(X)\U$.

\begin{theorem}[\cite{Abtahi-vector-valued}]
\label{thm:main-thm-holds-for-UA}
  If $\|\hat f\|=\|f\|_X$, for all $f\in \fA$, then every character $\psi\in\mis(\fA)$
  lifts to some $A$-character $\Psi\in\mis_A(\A)$. In particular, if $\fA$ is a uniform algebra,
  then $\A$ satisfies all conditions
  in Theorem $\ref{thm:every-psi-lifts-to-Psi-iff-every-f-extends-to-F}$.
\end{theorem}

\section{$C(X)$-valued Spectrum}
\label{sec:C(X)-valued Spectrum}

In this section, we study the vector-valued spectrum $\vsp(f)$ of an $A$-valued function
$f$ on $X$ defined by \eqref{eqn:intro:Sp(f)-dfn}

\begin{theorem}
\label{thm:SP(f)=phi o f}
  $\vsp(f)=\set{\phi\circ f : \phi \in \mis(A)}$.
\end{theorem}

\begin{proof}
  First, take $\phi\in \mis(A)$. We show that $\phi\circ f \in \vsp(f)$.
  If $\phi\circ f \notin \vsp(f)$, then, for some points $x_1, \dotsc, x_n\in X$
  and vectors $a_1,\dotsc, a_n \in A$, we have
  \[
    \U= \sum_{i=1}^n a_i(\phi\circ f(x_i)\U - f(x_i)).
  \]
  Therefore,
  \[
   1=\phi(\U)
    = \sum_{i=1}^n \phi(a_i)(\phi(f(x_i)) - \phi(f(x_i)))=0.
  \]
  This is a contradiction. Hence $\phi\circ f \in \vsp(f)$.

  Conversely, take $\lambda \in \vsp(f)$, and consider the ideal $I$ in $A$
  generated by
  \[
    S=\{\lambda(x)\U-f(x):x\in X\}.
  \]
  Since $\U\notin I$, the ideal $I$ is proper. So there is a character $\phi\in\mis(A)$
  with $I\subset \ker\phi$. This means that $\phi(\lambda(x)\U-f(x))=0$,  for all $x\in X$,
  whence $\lambda=\phi\circ f$.
\end{proof}

Next, we show that $\vsp(f)$, for $f\in C(X,A)$, is a compact subset of $C(X)$.
To this end, the following lemma is needed. Recall that a family $\cF$ of continuous
functions from $X$ into a metric space $(Y,d\,)$ is \emph{equicontinuous}
if, for every $x\in X$ and every $\e > 0$, there is a neighborhood
$U_x$ of $x\in X$ such that
\[
  d(f(x),f(y))<\e \quad (y\in U_x,\, f\in\cF).
\]

\begin{lemma}
\label{lem:pre-upper-semi-cnts-0}
  If a family $\cF$ of functions in $C(X,A)$ is equicontinuous,
  then the family $\cG$ defined by
  \begin{equation}\label{eqn:cG}
    \cG= \set{\phi\circ f: \phi\in \mis(A),\, f\in\cF}
  \end{equation}
  is equicontinuous.
\end{lemma}

\begin{proof}
  Let $x_0\in X$ and $\e>0$. Since $\cF$ is equicontinuous, there is a neighbourhood
  $U_0$ of $x_0$ such that $\|f(x)-f(x_0)\|<\e$, for all $f\in \cF$ and $x\in U_0$.
  Then
  \[
    |\phi\circ f(x)- \phi\circ f(x_0)|\leq \|f(x)-f(x_0)\|<\e
    \quad (\phi\in \mis(A),\,f\in \cF,\,x\in U_0).
  \]
  Hence $\cG$ is equicontinuous.
\end{proof}

The Arzel\'a-Ascoli theorem states that a family $\cF \subset C(X)$ is relatively compact
if $\cF$ is equicontinuous and pointwise bounded.

\begin{theorem}\label{thm:Sp(f) is compact in C(X)}
  For every $f\in C(X,A)$, the spectrum $\vsp(f)$ is a compact subset of $C(X)$.
\end{theorem}

\begin{proof}
  That $\vsp(f)\subset C(X)$ follows from Theorem \ref{thm:SP(f)=phi o f} and the assumption
  that $f:X\to A$ is continuous. To show that $\vsp(f)$ is compact in $C(X)$,
  using Arzel\'a-Ascoli theorem, we only need to show that $\vsp(f)$ is
  uniformly closed, uniformly bounded and equicontinuous.

  We prove that $\vsp(f)$ is uniformly closed in $C(X)$. Take a function
  $\lambda\in C(X)$ and assume that $\lambda\notin\vsp(f)$. Then, there exist $a_1,\dotsc,a_m \in A$ and
  $x_1, \dotsc, x_m \in X$ such that
  \begin{equation}\label{eqn:g notin SP(f)}
  \U=\sum_{i=1}^{m} a_i\bigprn{\lambda(x_{i})\U - f(x_{i})}.
  \end{equation}

  To get a contradiction, assume $\lambda\in \overline{\vsp(f)}$. Then, there is a sequence
  $\{\phi_n\}$ in $\mis(A)$ such that  $\phi_n\circ f\to \lambda$, uniformly on $X$.
  Take $\e=\prn{\|a_1\|+\dotsb+\|a_m\|}^{-1}$, where $a_1,\dotsc,a_m$ are
  given by \eqref{eqn:g notin SP(f)}.
  There is $N\in \N$, such that $\|\phi_n\circ f-\lambda\|_X <\e$, for all $n\geq N$. This implies that
  \[
    \abs{\phi_N(f(x_i))-\lambda(x_i)}< \e \quad (i=1,2,\dotsc,m).
  \]

  Now, by \eqref{eqn:g notin SP(f)}, we get
  \[
     1 = \phi_N(\U) = \sum_{i=1}^m \phi_N(a_i)\bigprn{\lambda(x_i)-\phi_N(f(x_i))}.
  \]
  Therefore,
  \[
    1 \leq \sum_{i=1}^m \bigabs{\phi_N(a_i)\bigprn{\lambda(x_i)-\phi_N(f(x_i))}}
      < \sum_{i=1}^m\|a_i\|\e = 1.
  \]
  This is a contradiction, and thus $\lambda\notin \overline{\vsp(f)}$.
  We conclude that $\vsp(f)$ is uniformly closed in $C(X)$.
  That $\vsp(f)$ is uniformly bounded follows from the fact that, for every $\phi\in\mis(A)$,
  \[
    \|\phi\circ f\|_X
    = \sup \set{|\phi(f(x))|:x\in X}
    \leq \|\phi\|\sup \set{\|f(x)\|:x\in X} = \|f\|_X.
  \]

  Finally, by taking $\cF=\set{f}$, Lemma \ref{lem:pre-upper-semi-cnts-0} implies that
  $\vsp(f)$ is equicontinuous.

  The set $\vsp(f)$ being uniformly closed, uniformly bounded and equicontinuous
  is a compact subset of $C(X)$.
\end{proof}

Now, analogous to \cite[Proposition 5.17]{BD}, we prove that the set-valued mapping
$f\mapsto \vsp(f)$ of $\A$ into the family of compact sets in $C(X)$
is upper semi-continuous.

\begin{lemma}
\label{lem:pre-upper-semi-cnts-1}
  Suppose $f_n\to f$ in $\A$, $\lambda_n\in\vsp(f_n)$ and $\lambda_n\to \lambda$ in $C(X)$.
  Then $\lambda\in \vsp(f)$.
\end{lemma}

\begin{proof}
  Towards a contradiction, assume $\lambda\notin \vsp(f)$. Then there exists a finite set
  of points $x_1,\dotsc,x_n$ in $X$ and vectors $a_1,\dotsc,a_n$ in $A$
  such that
  \begin{equation}\label{eqn:1=sum}
    \U=\sum_{i=1}^n a_i(\lambda(x_i)\U-f(x_i)).
  \end{equation}

  Take $\e=1/\bigprn{2\sum\|a_i\|}$. By the assumption, there is $N\in\N$
  such that
  \[
    \|f_N-f\|_X<\e, \quad \|\lambda_N-\lambda\|_X<\e.
  \]
  Since $\lambda_N\in \vsp(f_N)$, by Theorem \ref{thm:SP(f)=phi o f}, there is some
  $\phi\in \mis(A)$ such that $\lambda_N=\phi\circ f_N$. Using \eqref{eqn:1=sum}, we have
  \begin{align*}
    1 & = \Bigabs{\sum_{i=1}^n \phi(a_i)\bigprn{\lambda(x_i)-\phi\circ f(x_i)}} \\
    & \leq \sum_{i=1}^n \|a_i\|\bigabs{\lambda(x_i)-\lambda_N(x_i)+\phi\circ f_N(x_i)-\phi\circ f(x_i)} \\
    & \leq \sum_{i=1}^n \|a_i\|\bigprn{\|\lambda-\lambda_N\|_X+\|f_N-f\|_X}\\
    & < 2\e\sum_{i=1}^n \|a_i\| = 1.
  \end{align*}

  This is absurd. Hence $\lambda\in \vsp(f)$.
\end{proof}

\begin{theorem}
  The mapping $f\mapsto \vsp(f)$ of $\A$ into the family of compact sets
  in $C(X)$ is upper semi-continuous.
\end{theorem}

\begin{proof}
  Suppose the mapping is not upper semi-continuous at some $f_0\in\A$. Hence, there is
  a neighbourhood $\Omega$ of $\vsp(f_0)$ in $C(X)$ and a sequence $\{f_n\}$ in
  $\A$ such that $f_n\to f_0$ and $\vsp(f_n)\not\subset \Omega$.
  Set $\cF=\set{f_0,f_1,f_2,\dotsc}$ and define $\cG$ as in \eqref{eqn:cG}.
  Since $f_n\to f_0$ uniformly on $X$, the family $\cF$ is equicontinuous
  on $X$. By Lemma \ref{lem:pre-upper-semi-cnts-0}, the family $\cG$ is equicontinuous.
  Since $\cF$ is uniformly bounded, $\cG$ is uniformly bounded, and, therefore, relatively compact.

  Now, consider the mapping $\vsp:\cF\to2^\cG$, $f\mapsto \vsp(f)$.
  By Lemma \ref{lem:pre-upper-semi-cnts-1} and \cite[Lemma 5.16]{BD},
  this mapping is upper semi-continuous at $f_0$, which contradicts our
  primary assumption.
\end{proof}

\subsection*{The $\fA$-valued spectrum}
  When $\A$ is an admissible $A$-valued function algebra on $X$
  and $\fA=\A\cap C(X)$, then $\vsp(f) \subset \fA$, for every $f\in \A$.
  In this case, it is convenient to denote the vector-valued spectrum
  by $\vsp_\fA(f)$ and call it the \emph{$\fA$-valued spectrum}.
  If the mapping $\tilde f:\mis(A)\to \fA$, $\phi\mapsto \phi\circ f$,
  is continuous (which is the case when $\fA$ is a uniform algebra) then $\vsp_\fA(f)$
  is compact in $\fA$ since $\vsp_\fA(f) = \tilde f(\mis(A))$. In general,
  however, it is unknown if $\tilde f$ is continuous or $\vsp_\fA(f)$
  is compact in $\fA$. We remark that $\vsp_\fA(f)$ is always closed in $\fA$.

\section{The $A$-valued Spectrum}

Let $\A$ be an admissible Banach $A$-valued function algebra on $X$.
Take $f\in \A$ and consider the function $\tilde f:\mis(A)\to \fA$,
$\phi\mapsto\phi\circ f$.
By Theorem \ref{thm:SP(f)=phi o f}, we have
\[
  \vsp(\tilde f) = \set{\psi\circ \tilde f: \psi\in \mis(\fA)}.
\]

Note that every $\lambda\in\vsp(\tilde f)$ is a function of $\mis(A)$ into $\C$.
Also, every $a\in A$ induces a function $\hat a:\mis(A)\to \C$,
the Gelfand transform of $a$.

\begin{definition}
  The \emph{$A$-valued spectrum} of a function $f\in \A$ is defined by
  \begin{equation}\label{eqn:SPA(f)}
    \vsp_A(f) = \set{a\in A : \hat a \in \vsp(\tilde f)}.
  \end{equation}
\end{definition}

\begin{theorem}\label{thm:SPA(f)=Psi(f)}
  For every $f\in \A$, we have
  \begin{equation}\label{eqn:f(X) subset Psi(f) subset SPA(f)}
    f(X) \subset \set{\Psi(f):\Psi\in\mis_A(\A)} \subset \vsp_A(f).
  \end{equation}
  If every character $\psi\in\mis(\fA)$ lifts to some $A$-character $\Psi\in \mis_A(\A)$,
  then
  \begin{equation}\label{eqn:SPA(f)=Psi(f)}
    \vsp_A(f)=\set{\Psi(f):\Psi\in\mis_A(\A)}.
  \end{equation}
\end{theorem}

Equality \eqref{eqn:SPA(f)=Psi(f)} is an analogue
of \eqref{eqn:intro:sp(a)=phi(a)}.

\begin{proof}
  The first inclusion in \eqref{eqn:f(X) subset Psi(f) subset SPA(f)} is obvious
  since, for every $x\in X$, the evaluation homomorphism $\cE_x:f\mapsto f(x)$ is
  an $A$-character and $f(X) = \set{\cE_x(f):x\in X}$.

  To prove the second inclusion in \eqref{eqn:f(X) subset Psi(f) subset SPA(f)},
  take an $A$-character $\Psi\in\mis_A(\A)$ and let $a=\Psi(f)$.
  Take a finite set $\phi_1,\dotsc,\phi_n$ of elements in $\mis(A)$ and
  $g_1,\dotsc,g_n$ of functions in $\fA$, and consider the function
  \[
    g = \sum_{i=1}^n g_i(\hat a(\phi_i)\U-\phi_i\circ f).
  \]

  \noindent
  Let $\psi=\Psi|_\fA$. Then $\psi\in\mis(\fA)$ and thus $\psi(\U)=1$. We have
  \begin{align*}
    \psi(g)=\psi\biggprn{\sum_{i=1}^n g_i(\hat a(\phi_i)\U-\phi_i\circ f)}
     & = \sum_{i=1}^n \psi(g_i)(\hat a(\phi_i)-\psi(\phi_i\circ f)) \\
     & = \sum_{i=1}^n \psi(g_i)(\phi_i(a) - \phi_i(\Psi f)) \\
     & = \sum_{i=1}^n \psi(g_i)(\phi_i(a) - \phi_i(a)) = 0.
  \end{align*}
  This implies that $g\neq\U$. Hence $a\in\vsp_A(f)$, that is, $\Psi(f)\in\vsp_A(f)$.

  Now, assume that every character $\psi\in\mis(\fA)$ lifts to some $A$-character
  $\Psi\in \mis_A(\A)$. Take an element $a\in\vsp_A(f)$. We show the existence
  of some $A$-character $\Psi$ such that $a = \Psi(f)$. Since $a\in\vsp_A(f)$,
  we have $\hat a\in\vsp(\tilde f)$. Hence there is a character $\psi\in\mis(\fA)$
  such that $\hat a = \psi\circ \tilde f$. Assume $\psi$ lifts to the $A$-character
  $\Psi$ of $\A$. Then, for every $\phi\in\mis(A)$,
  \begin{equation*}
    \phi(a) = \hat a(\phi) = \psi(\tilde f(\phi))= \psi(\phi\circ f) = \phi(\Psi(f)).
  \end{equation*}

  Since $A$ is assumed to be semisimple, we get $a=\Psi(f)$.
\end{proof}

\begin{remark}
  If every character $\psi\in\mis(\fA)$ lifts to some $A$-character
  $\Psi\in \mis_A(\A)$, then by
  Theorem \ref{thm:every-psi-lifts-to-Psi-iff-every-f-extends-to-F}
  each $f\in\A$ extends to a function $F:\mis(\fA)\to A$.
  In this case, $\vsp_A(f)=F(\mis(\fA))$ which is compact if $F$ is continuous.
\end{remark}

\begin{corollary}
\label{cor:SP(f)=f(X)}
  If $\fA=C(X)\cap\A$ is a natural function algebra on $X$, then
  $\vsp_A(f)=f(X)$, for all $f\in \A$. In particular,
  $\vsp_A(f)$ is compact in $A$.
\end{corollary}

\begin{proof}
  Since $\fA$ is natural, the only characters of $\fA$ are the point
  evaluation homomorphisms $\e_x\ (x\in X)$. Therefore, $\cE_x\ (x\in X)$
  are the only $A$-characters of $\A$. From \eqref{eqn:SPA(f)=Psi(f)},
  we get $\vsp_A(f)=f(X)$ for all $f\in \A$.
\end{proof}

\section{Examples}

In this section, we identify the $A$-valued spectrum in certain
algebras.

\begin{example}\label{exa:SP(f)=f(X)}
  Let $\A=C(X,A)$. Then $\fA=C(X)$ is natural.
  By Corollary \ref{cor:SP(f)=f(X)},
  we have $\vsp_A(f)=f(X)$, for all $f\in C(X,A)$.
\end{example}

\begin{example}
\label{exa:LIP}
  Let $(X,\rho)$ be a compact metric space. An $A$-valued \emph{Lipschitz function}
  is a function $f : X \to A$ such that
  \begin{equation}
    L(f)=\sup\biggset{\frac{\norm{f(x)-f(y)}}{\rho(x,y)}:x,y\in X,\, x\neq y}<\infty.
  \end{equation}

  Denoted by $\Lip(X,A)$, the space of $A$-valued Lipschitz functions on $X$ is
  an $A$-valued function algebra on $X$, called the \emph{$A$-valued Lipschitz algebra}.
  For $f\in \Lip(X,A)$, the Lipschitz norm of $f$ is defined by
  $\norm{f}_L = \norm{f}_X + L(f)$. It is easily verified that $(\Lip(X,A),\enorm_L)$
  is an admissible Banach $A$-valued function algebra and $\Lip(X)=\Lip(X,A)\cap C(X)$,
  where $\Lip(X)=\Lip(X,\C)$ is the classic complex-valued Lipschitz algebra.
  It is proved in \cite{Sherbert-63} that
  $\Lip(X)$ is natural. Therefore, by Corollary \ref{cor:SP(f)=f(X)},
  we have $\vsp_A(f)=f(X)$, for all $f\in \Lip(X,A)$.
\end{example}

\begin{example}
  Let $K$ be a compact set in the complex plane.
  Let $P_0(K,A)$ be the algebra of the restriction to $K$
  of all polynomials $p(z)=a_nz^n+\dotsb+a_1z+a_0$ with coefficients
  $a_0,a_1,\dotsc,a_n$ in $A$. Let $R_0(K,A)$ be the algebra
  of the restriction to $K$ of all rational functions of the form
  $p(z)/q(z)$, where $p(z)$ and $q(z)$ are polynomials with coefficients
  in $A$, and $q(\lambda) \in \Inv(A)$, whenever $\lambda\in K$.

  The algebras $P_0(K,A)$ and $R_0(K,A)$ are $A$-valued function algebras
  on $K$, and their uniform closures in $C(K,A)$, denoted by $P(K,A)$ and $R(K,A)$,
  are $A$-valued uniform algebras. When $A=\C$, we drop $A$ and write $P(K)$ and $R(K)$,
  which are complex uniform algebras.

  For $P(K,A)$, we have $\fA=P(K)$ and it is known
  that the character space of $P(K)$ is naturally identified with $\hat K$
  the polynomially convex hull of $K$ (\cite{Gamelin-UA}). Since $\fA$ is
  a uniform algebra, by
  Theorems \ref{thm:every-psi-lifts-to-Psi-iff-every-f-extends-to-F} and
  \ref{thm:main-thm-holds-for-UA}, every $f\in P(K,A)$ extends to a function
  $F:\hat K\to A$. We see that, for all $f\in P(K,A)$,
  \[
    \vsp_A(f)=\set{\Psi(f):\Psi\in\mis_A(\A)}
    =\set{F(\psi):\psi\in\mis(\fA)}=F(\hat K)
  \]

  For $R(K,A)$, we have $\fA=R(K)$ and it is known
  that $R(K)$ is natural (\cite{Gamelin-UA}). By Corollary \ref{cor:SP(f)=f(X)},
  we have
  \[
    \vsp_A(f) = f(K) \quad (f\in R(K,A)).
  \]
\end{example}

The final example shows that the notion of $A$-valued spectrum coincides with
the notion of Waelbroeck spectrum \cite[Definition 6]{Dineen-Harte-Taylor-I}.

\begin{example}[Tensor Products]
\label{exa:Tensor Products}
For a Banach function algebra $\fA$ on $X$, consider
the algebraic tensor product $\fA\tp A$. By \cite[Theorem 42.6]{BD},
there is a linear mapping $T:\fA\tp A\to \fA A$ such that
\begin{equation}\label{eqn:T:fA tp A to C(X,A)}
  T \Bigprn{\sum_{i=1}^n f_i \tp a_i} = \sum_{i=1}^n f_i a_i.
\end{equation}

The mapping $T$ is, in fact, an algebra isomorphism so that 
the tensor product $\fA\tp A$ can be seen as an admissible $A$-valued
function algebra on $X$. The mapping $T$ in \eqref{eqn:T:fA tp A to C(X,A)}
extends to an isometric isomorphism of the injective tensor product $\fA\ptp_\epsilon A$
onto the uniform closure $\overline{\fA A}$ of $\fA A$ (see \cite{Abtahi-vector-valued}).
We identify every $f\in \fA\tp A$ with its image $Tf$ given by \eqref{eqn:T:fA tp A to C(X,A)}.
If $\enorm_\epsilon$ denote the injective tensor norm, then
$\|f\|_\epsilon = \|f\|_X$, for all $f\in \fA\tp A$.

In general, let $\enorm_\gamma$ be an algebra cross-norm on $\fA\tp A$,
and let $\fA\ptp_\gamma A$ be its completion.

\begin{lemma}
\label{lem:fA tp A is a FA}
  $\fA\ptp_\gamma A$ is a Banach $A$-valued function algebra on $X$.
\end{lemma}

\begin{proof}
 We need to show that every $f\in \fA\ptp_\gamma A$ is a continuous function
 of $X$ into $A$ and $\|f\|_X \leq \|f\|_\gamma$. Given $\e>0$, the element
 $f$ has a representation of the form $f= \sum_{n=1}^\infty f_n$,
 with $f_n \in \fA\tp A$, and
 \[
   \|f\|_\gamma \leq \sum_{n=1}^\infty \|f_n\|_\gamma <\|f\|_\gamma+\e.
 \]

 Since $\enorm_\gamma$ is a cross-norm, we have
 $\|f_n\|_X=\|f_n\|_\epsilon \leq \|f_n\|_\gamma$, for all $n\geq1$, and
 thus the series $\sum f_n$ converges uniformly, whence
 $f$ is a continuous function of $X$ into $A$ and $\|f\|_X\leq \|f\|_\gamma$.
\end{proof}

\begin{lemma}
\label{lem:fA ptp A is admissible}
 If $f\in \fA \ptp_\gamma A$ and $\phi\in A^*$ then $\phi\circ f\in \fA$
 and $\|\phi\circ f\| \leq \|\phi\|\|f\|_\gamma$.
 In particular, the algebra $\fA \ptp_\gamma A$ is admissible.
\end{lemma}

\begin{proof}
  First, assume $f=\sum_{i=1}^n f_i\tp a_i\in A\tp\fA$.
  Then, for every $\phi\in A^*$,
  \[
    \phi\circ f= \sum_{i=1}^n \phi(a_i) f_i \in \fA.
  \]

  \noindent
  Let $\fA^*_1$ and $A^*_1$ denote the closed unit ball of $\fA^*$ and
  $A^*$, respectively. In case $\|\phi\|=1$, we have
  \begin{equation}
  \begin{split}
    \|\phi\circ f\|
     & = \sup\set{|\psi(\phi\circ f)|:\psi\in \fA^*_1} \\
     & \leq \sup\set{|\psi(\vp\circ f)|:\psi\in \fA^*_1,\, \vp\in A^*_1} \\
     & = \sup\Bigset{\Bigabs{\sum_{i=1}^n \psi(f_i)\vp(a_i)}:\psi\in \fA^*_1,\, \vp\in A^*_1} \\
     & = \|f\|_\epsilon \leq \|f\|_\gamma.
  \end{split}
  \end{equation}
  In general case, we have
  $\|\phi\circ f\| \leq \|\phi\|\|f\|_\gamma$, for all $\phi\in A^*$ and $f\in \fA\tp A$.

  Next, consider $f\in \fA\ptp_\gamma A$. There is a sequence $\{f_n\}$ in $\fA\tp A$
  that converges to $f$ with respect to $\enorm_\gamma$. Since
  $\|\phi\circ f_n-\phi\circ f_m\| \leq \|\phi\|\|f_n-f_m\|_\gamma$, for all $m,n$,
  the sequence $\{\phi\circ f_n\}$ is Cauchy in $\fA$ whence it converges to
  some $h\in \fA$. Since $\|f_n-f\|_X\leq \|f_n-f\|_\gamma$,
  we have $f_n(x)\to f(x)$, for all $x\in X$, and thus
  \[
    h(x)= \lim_{n\to\infty} \phi(f_n(x)) = \phi(f(x)) = \phi\circ f(x).
  \]

  \noindent
  We see that $h=\phi\circ f\in \fA$ and
  \[
    \|\phi\circ f\| = \lim_{n\to\infty}\|\phi\circ f_n\|
    \leq \|\phi\|\lim_{n\to\infty}\|f_n\|_\gamma = \|\phi\|\|f\|_\gamma.
  \qedhere
  \]
\end{proof}

\begin{lemma}
  Let $f\in \fA \ptp_\gamma A$ and $\psi\in \mis(\fA)$.
  \begin{enumerate}[\quad$(1)$]
    \item \label{item:one}
    With respect to the weak* topology of $A^*$, the mapping 
    $\tilde f:A^* \to \fA$ given by $\tilde f(\phi)=\phi\circ f$, is continuous 
    on bounded subsets of $A^*$. 
    
    \item \label{item:two}
    The mapping $g:A^*\to \C$, $\phi\mapsto\psi(\phi\circ f)$ is continuous.
    
    \item \label{item:three}
    The character $\psi$ lifts to some $A$-character $\Psi\in\mis_A(\A)$.
  \end{enumerate}

\end{lemma}

\begin{proof}
  \eqref{item:one} 
  First, assume $f=\sum_{i=1}^n f_i\tp a_i\in \fA\tp A$.
  Given $\phi_0 \in A^*$ and $\e>0$, define a neighborhood $U_0$ of $\phi_0$
  as follows:
  \[
    U_0=\Bigset{\phi \in A^*: |\phi(a_i)-\phi_0(a_i)|
      < \frac\e{\sum_{i=1}^n\|f_i\|}, 1\leq i\leq n}.
  \]

  \noindent
  Then, for each $\phi \in U_0$,
  \[
    \|\tilde f(\phi)-\tilde f(\phi_0)\|
      = \|\phi\circ f - \phi_0\circ f\|
       \leq \sum_{i=1}^n \|f_i\||\phi(a_i)-\phi_0(a_i)| < \e.
  \]
  We see that, in this case, $\tilde f$ is continuous on $A^*$.

  In general case where $f\in \fA\ptp_\gamma A$, take, for every $\e>0$, 
  an element $f_0\in \fA\tp A$
  such that $\|f - f_0\|_\gamma < \e$. Assume $\{\phi_\alpha\}$ is a bounded net in $A^*$
  that converges to $\phi_0$, in the weak* topology. Suppose $\|\phi_\alpha\| \leq M$,
  for all $\alpha$. Since $\tilde f_0:A^*\to \fA$
  is continuous, there is $\alpha_0$ such that $\|\tilde f_0(\phi_\alpha)-\tilde f_0(\phi_0)\|<\e$,
  for $\alpha\geq \alpha_0$. Now, for $\alpha\geq \alpha_0$, we have
  \begin{align*}
    \|\tilde f(\phi_\alpha)-\tilde f(\phi_0)\|
      & \leq \|\phi_\alpha\circ f - \phi_\alpha\circ f_0\|
           + \|\phi_\alpha\circ f_0 - \phi_0\circ f_0\|+\|\phi_0\circ f_0 - \phi_0\circ f\|\\
      & \leq \|\phi_\alpha\|\|f-f_0\|_\gamma + \e + \|\phi_0\|\|f - f_0\|_\gamma \\
      & < (M+\|\phi_0\|+1)\e.
   \end{align*}

   \noindent
   This implies that $\tilde f$ is continuous on bounded subsets of $A^*$.

   \eqref{item:two} follows from Corollary 3.11.4 in \cite{Horvath}
   and the fact that $g$ is continuous on bounded subsets of $A^*$.
   \eqref{item:three} follows from \eqref{item:two} and
   Theorem \ref{thm:every-psi-lifts-to-Psi-iff-every-f-extends-to-F}.
\end{proof}

That every $\psi\in \mis(\fA)$ lifts to some $\Psi\in \mis_A(\fA\ptp_\gamma A)$
can be verified explicitly as follows. First, define a mapping
\[
  \Psi_0: \fA\tp A\to A,\quad \Psi_0\Bigprn{\sum_{i=1}^n f_i\tp a_i} = \sum_{i=1}^n \psi(f_i)a_i.
\]

\noindent
The mapping $\Psi_0$ is an algebra homomorphism with
$\phi(\Psi_0(f))=\psi(\phi\circ f)$, for all $\phi\in A^*$.
It is easily verified that $\|\Psi_0(f)\| \leq \|f\|_\gamma$,
for all $f\in \fA\tp A$. Being a continuous homomorphism on
a dense subspace of $\fA\ptp_\gamma A$, $\Psi_0$ can extend to a
homomorphism $\Psi:\fA\ptp_\gamma A \to A$ with the desired properties.

We summarize the above discussion in the following statement.

\begin{theorem}
\label{thm:fA tp A is a FA}
  The algebra $\fA \ptp_\gamma A$ is an admissible Banach $A$-valued function algebra
  on $X$ that satisfies all conditions in
  Theorem $\ref{thm:every-psi-lifts-to-Psi-iff-every-f-extends-to-F}$.
\end{theorem}

In \cite{Dineen-Harte-Taylor-I}, the Waelbroeck spectrum of elements
$f\in \fA\ptp_\gamma A$ has been studied. An examination of the results 
shows that, given $\psi\in \mis(\fA)$, the mapping $\psi\tp I_A$ there
coincides with the $A$-character $\Psi$ to which $\psi$ lifts. Hence,
the Waelbroeck spectrum $\sigma_W(f)$ of every $f\in \fA\ptp_\gamma A$ \cite[Definition 6]{Dineen-Harte-Taylor-I} coincides with the $A$-valued spectrum $\vsp_A(f)$.
That is
\[
  \sigma_W(f)=\vsp_A(f) = \set{\Psi(f):\Psi\in \mis_A(\fA \ptp_\gamma A)}.
\]
\end{example}


\bibliographystyle{amsplain}

\begin{thebibliography}{99}

\bibitem{Abtahi-vector-valued}
 Mortaza. Abtahi,
 \textit{Vector-valued characters on vector-valued function algebras},
 Banach J. Math. Anal. (to appear), \texttt{arXiv:1509.09215}.


\bibitem{Arens-Calderon}
  R. Arens and A.P. Calder\'on,
  \textit{Analytic functions of several Banach algebra elements},
  Ann. of Math. \textbf{62} no. 2, 204--216, (1955).

\bibitem{BD}
 F. F. Bonsall, J. Duncan,
 \textit{Complete Normed Algebras},
 Springer-Verlag, Berlin, Heidelberg, New York, (1973).


\bibitem{Dineen-2004}
 S. Dineen,
 \textit{Spectral theory, tensor products and infinite dimensional holomorphy},
 J. Korean Math. Soc. \textbf{41} No. 1, 193--207, (2004).

\bibitem{Dineen-Harte-Taylor-I}
 S. Dineen, R.E. Harte, C. Taylor,
 \textit{Spectra of tensor product elements I: basic theory},
 Math. Proc. Royal Irish Academy, \textbf{101}A (2), 177--196, (2001).

\bibitem{Dineen-Harte-Taylor-II}
 S. Dineen, R.E. Harte, C. Taylor,
 \textit{Spectra of tensor product elements II: polynomials},
 Math. Proc. Royal Irish Academy, \textbf{101}A (2), 197--220, (2001).

\bibitem{Dineen-Harte-Taylor-III}
 S. Dineen, R.E. Harte, C. Taylor,
 \textit{Spectra of tensor product elements III: holomorphic},
 Math. Proc. Royal Irish Academy, \textbf{103}A (1), 61--92, (2003).

\bibitem{Gamelin-UA}
 T.W. Gamelin,
 \textit{Uniform Algebras},
 Prentice-Hall, Inc., Englewood Cliffs, NJ., (1969).

\bibitem{Harte-1973}
 R.E. Harte,
 \textit{The spectral mapping theorem in many variables}, in I.G. Craw and A.J. White (eds),
 Proceedings of the Seminar `Uniform Algebras', University of Aberdeen, 1973,
 Department of Mathematics, University of Aberdeen, Aberdeen, 1973, pp 59--63.

\bibitem{Harte-gmj}
 R. Harte, C. Taylor,
 \textit{On Vector-valued Spectra},
 Glasgow Math. J. \textbf{42}, 247--253, (2000).

\bibitem{Horvath}
   J. Horv\'ath,
   \textit{Topological vector spaces and distributions}, Vol. I,
   Addison-Wesley, Reading, Mass., (1966).

\bibitem{Matos-1984}
  M.C. Matos,
  \textit{Approximation of analytic functions by rational functions in Banach spaces},
  J. Funct. Anal. \textbf{56} no. 2, 251--264, (1984).

\bibitem{Matos-1988}
  M.C. Matos,
  \textit{On holomorphy in Banach spaces and absolute convergence of Fourier series},
  Port. Math. \textbf{45} no. 4, 429--450, (1988).


\bibitem{Sherbert-63}
 D.R. Sherbert,
 \textit{Banach algebras of Lipschitz functions},
 {Pacific J. Math.}, \textbf{13} (1963), 1387--1399.

\bibitem{Shilov-1953}
  G.E. Silov,
  \textit{On the decomposition of a commutative normed ring into a direct sum of ideals},
  Matematicheskii Sbornik, \textbf{32}(74), 353--364, (1953). Translated in
  Amer. Math. Soc. Transl., 1, 2, (1955).

\bibitem{Waelbroeck-1954}
 L. Waelbroeck,
 \textit{Le calcul symbolique dans les alg\'ebres commutatives},
 J.Math. Pures Appl. \textbf{33} 147--186, (1954).


\end{thebibliography}

\end{document}